\documentclass[10pt]{elsarticle}
\usepackage{amssymb}
\usepackage{amsmath}
\usepackage{amsfonts}
\usepackage{amsbsy}
\usepackage{amscd}
\usepackage{amsthm}
\usepackage{fullpage}
\usepackage{times}
\usepackage{psfrag}
\usepackage{graphics}
\usepackage{color}
\usepackage{bbm}
\usepackage{colortbl}
\usepackage{graphicx}
\usepackage{caption}
\usepackage{bbm}
\usepackage{subcaption}
\usepackage{array,supertabular}
\usepackage{tabularx}
\usepackage{booktabs}
\usepackage{multirow}
\usepackage{ulem}
\usepackage{units}
\usepackage{mathabx}
\usepackage{accents}
\usepackage{algorithmicx}
\usepackage{relsize}
\usepackage{algorithm}
\usepackage{xfrac}
\usepackage{algpseudocode}
\usepackage{rotating}
\usepackage[capitalize]{cleveref}
\usepackage{mathrsfs} 
\usepackage[framemethod=tikz]{mdframed} 
\usepackage{circuitikz}

\crefname{secinapp}{Section}{Sections}
\Crefname{secinapp}{Section}{Sections}

\newlength{\dhatheight}

\newtheorem{theorem}{Theorem}[section]
\newtheorem{corollary}[theorem]{Corollary}

\newtheorem{definition}[theorem]{Definition}
\newtheorem{remark}[theorem]{Remark}
\newtheorem{lemma}[theorem]{Lemma}
\newtheorem{proposition}[theorem]{Proposition}
\newtheorem*{main1}{Theorem \ref{thm:bend}}
\newtheorem*{old}{Theorem}

\begin{document}

\begin{frontmatter}
\title{Resistance distance in bent linear 2-trees}
\author[byu1]{Wayne Barrett}
\author[byu1]{Emily. J. Evans}
\author[carroll]{Amanda E. Francis}

\address[byu1]{Department of Mathematics,
  Brigham Young University,
  Provo, Utah 84602, USA}
\address[carroll]{Department of Mathematics, Engineering, and Computer Science, 
  Carroll College,
  Helena, Montana 59601, USA}

\begin{abstract}
In this note we consider the bent linear 2-tree and provide an explicit formula for the resistance distance $r_{G_n}(1,n)$ between the first and last vertices of the graph.  We call the graph $G_n$ with vertex set $V(G_n) = \{ 1, 2, \ldots, n\}$ and $\{i,j\} \in E(G_n)$ if and only if $0<|i-j| \leq 2$ a straight linear 2-tree.  We define the graph $H_n$ with $V(H_n)= V(G_n)$ and $E(H_n) = (E(G_n) \cup \{k,k+3\})\setminus(\{k+1,k+3\}$ to be a bent linear 2-tree with bend at vertex $k$.  

\end{abstract}

\begin{keyword}
effective resistance, resistance distance, 2--tree
\end{keyword}

\end{frontmatter}

\section{Introduction}

An undirected graph $G$ consists of a finite set $V$ called the vertices, a subset $E$ of two-element subsets of $V$ called the edges, and a set $w$ of positive weights associated with each edge in $G$.  
For convenience we usually take the vertex set $V$ to be $\{1, 2, \ldots, n\}$. 
The adjacency matrix of $G$, $A(G)$, is the $n\times n$ nonnegative symmetric matrix defined by 
\[
a_{ij} = \left\{ \begin{array}{cl} w(i,j) & \text{if } \{i,j\} \text{ is an edge of }G\\
0 & \text{otherwise} \end{array}\right. 
\]
and its Laplacian matrix, $L(G)$ is the $n\times n$ real symmetric matrix defined by 
\[
\ell_{ij} = \left\{ \begin{array}{cl} \deg(i) & \text{if } i = j \\
-w(i,j) & \text{if } i \neq j \text{ and } \{i,j\} \text{ is an edge of } G\\
0 & \text{otherwise,}\end{array}\right.
\]
where $\deg(i)$ is the sum of the edge weights of the edges incident to $i$. 

In this note we will consider the resistance distance (or effective resistance) in a so-called ``bent linear 2-tree''.  Roughly speaking the resistance distance between two nodes on a graph is determined by considering the graph as an electrical circuit with edges being represented as resistors.  The resistance on each edge is the reciprocal of the edge weight.  Using the standard laws of electrical conductance the resistance distance of the graph can be determined.  For undirected graphs, a well known result links the resistance distance between two nodes to the Moore-Penrose inverse of the graph Laplacian.  Specifically given two nodes $i$ and $j$ of a graph the effective resistance $r_G(i,j)$ between the nodes is given by
\begin{equation}
r_G(i,j) = (\mathbf{e}_i - \mathbf{e}_j)^T L^\dagger (\mathbf{e}_i - \mathbf{e}_j).
\end{equation}
This formula works well for small and intermediate sized graphs, but is difficult to apply both theoretically and computationally for very large graphs.  Other techniques, such as circuit transformations can also be applied to determine the resistance distance.  For a few graphs resistance distance between any two vertices has been calculated explicitly as a function of the number of vertices in the graph~\cite{BapatWheels, YangKlein,bef}.  

Before defining a bent linear 2-tree we recall the definition of a linear 2-tree and a straight linear 2-tree.
\begin{definition}[linear 2-tree]\label{def:lin2tree}
A linear 2-tree is a graph $G$ that is constructed inductively by starting with a triangle and connecting each new vertex to the vertices of an existing edge that includes a vertex of degree 2.  
\end{definition}

We can also define a straight linear 2-tree as a special case of a linear 2-tree.
\begin{definition}[straight linear 2-tree]\label{def:lin2treest}
A straight linear 2-tree is a graph $G_n$ with $n$ vertices with adjacency matrix that is symmetric, banded, with the first and second subdiagonals equal to one and first and second superdiagonals equal to one, and all other entries equal to zero. See Figure~\ref{fig:2tree}. \end{definition} 

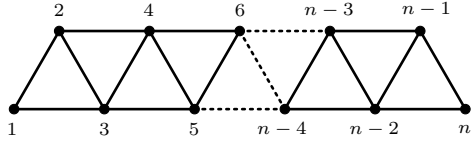
\begin{figure}[h!]
\begin{center}

\begin{tikzpicture}[line cap=round,line join=round,>=triangle 45,x=1.0cm,y=1.0cm,scale = 1.2]
\draw [line width=1.pt] (-3.,0.)-- (-2.,0.);
\draw [line width=1.pt] (-2.,0.)-- (-1.,0.);
\draw [line width=1.pt,dotted] (-1.,0.)-- (0.,0.);
\draw [line width=1.pt] (0.,0.)-- (1.,0.);
\draw [line width=1.pt] (1.,0.)-- (2.,0.);
\draw [line width=1.pt] (2.,0.)-- (1.5,0.866025403784435);
\draw [line width=1.pt] (1.5,0.866025403784435)-- (1.,0.);
\draw [line width=1.pt] (1.,0.)-- (0.5,0.8660254037844366);
\draw [line width=1.pt] (0.5,0.8660254037844366)-- (0.,0.);
\draw [line width=1.pt,dotted] (0.,0.)-- (-0.5,0.8660254037844378);
\draw [line width=1.pt] (-0.5,0.8660254037844378)-- (-1.,0.);
\draw [line width=1.pt] (-1.,0.)-- (-1.5,0.8660254037844385);
\draw [line width=1.pt] (-1.5,0.8660254037844385)-- (-2.,0.);
\draw [line width=1.pt] (-2.,0.)-- (-2.5,0.8660254037844388);
\draw [line width=1.pt] (-2.5,0.8660254037844388)-- (-3.,0.);
\draw [line width=1.pt] (-2.5,0.8660254037844388)-- (-1.5,0.8660254037844385);
\draw [line width=1.pt] (-1.5,0.8660254037844385)-- (-0.5,0.8660254037844378);
\draw [line width=1.pt,dotted] (-0.5,0.8660254037844378)-- (0.5,0.8660254037844366);
\draw [line width=1.pt] (0.5,0.8660254037844366)-- (1.5,0.866025403784435);
\begin{scriptsize}
\draw [fill=black] (-3.,0.) circle (1.5pt);
\draw[color=black] (-3.02279181666165,-0.22431183338253265) node {$1$};
\draw [fill=black] (-2.,0.) circle (1.5pt);
\draw[color=black] (-2.0001954862580344,-0.22395896857501957) node {$3$};
\draw [fill=black] (-2.5,0.8660254037844388) circle (1.5pt);
\draw[color=black] (-2.5018465162673555,1.100526386601008717) node {$2$};
\draw [fill=black] (-1.5,0.8660254037844385) circle (1.5pt);
\draw[color=black] (-1.5081915914412003,1.100526386601008717) node {$4$};
\draw [fill=black] (-1.,0.) circle (1.5pt);
\draw[color=black] (-1.0065405614318794,-0.22290037415248035) node {$5$};
\draw [fill=black] (-0.5,0.8660254037844378) circle (1.5pt);
\draw[color=black] (-0.4952423962300715,1.100526386601008717) node {$6$};
\draw [fill=black] (0.,0.) circle (1.5pt);
\draw[color=black] (-0.03217990699069834,-0.22431183338253265) node {$n-4$};
\draw [fill=black] (0.5,0.8660254037844366) circle (1.5pt);
\draw[color=black] (0.4887653934035965,1.103344389715898) node {$n-3$};
\draw [fill=black] (1.,0.) circle (1.5pt);
\draw[color=black] (0.9904164234129174,-0.22431183338253265) node {$n-2$};
\draw [fill=black] (1.5,0.866025403784435) circle (1.5pt);
\draw[color=black] (1.5692445349621338,1.1042991524908385) node {$n-1$};
\draw [fill=black] (2.,0.) circle (1.5pt);
\draw[color=black] (1.993718483431559,-0.22431183338253265) node {$n$};
\end{scriptsize}
\end{tikzpicture}
%
%
\end{center}
\caption{A straight linear 2-tree}
\label{fig:2tree}
\end{figure}
Earlier work by the authors~\cite{bef} considered resistance distance in a straight linear 2-tree with $n$ vertices and $m=n-2$ triangles and obtained the following result.
\begin{old}
Let $G_n$ be the straight linear 2-tree on $n$ vertices labeled as in Figure \ref{fig:2tree}, and let $m = n-2$ be the number of triangles in $G_n$.  Then for any two vertices $j$ and $j+k$ of $G_n$, 
\[
r_m(j,j+k) = \frac{\sum_{i=1}^{k} \left( F_iF_{i+2j-2} - F_{i-1}F_{i+2j-3}\right)F_{2m-2i-2j+5}}{F_{2m+2}},
\] 
\noindent where $F_p$ is the $p$th Fibonacci number. 
\end{old}

In this work we consider a modification of the straight linear 2-tree which we term the bent linear 2-tree whose definition is below.
\begin{definition}[bent linear 2-tree]\label{def:lin2treestb}
We define the graph $H_n$ with $V(H_n)= V(G_n)$ and $E(H_n) = (E(G_n) \cup \{k,k+3\})\setminus(\{k+1,k+3\}$ to be a bent linear 2-tree with bend at vertex $k$.
 See Figure~\ref{fig:bent}.
\end{definition}
\begin{figure}
\begin{center}
\begin{tikzpicture}[line cap=round,line join=round,>=triangle 45,x=1.0cm,y=1.0cm,scale = 1.2]
\draw [line width=.8pt] (-5.464101615137757,2.)-- (-4.598076211353318,1.5);
\draw [line width=.8pt] (-4.598076211353318,1.5)-- (-4.598076211353318,2.5);
\draw [line width=.8pt] (-4.598076211353318,2.5)-- (-3.732050807568879,2.);
\draw [line width=.8pt] (-3.732050807568879,2.)-- (-3.7320508075688785,3.);
\draw [line width=.8pt,dotted] (-3.7320508075688785,3.)-- (-2.8660254037844393,2.5);
\draw [line width=.8pt] (-2.8660254037844393,2.5)-- (-2.866025403784439,3.5);
\draw [line width=.8pt] (-2.866025403784439,3.5)-- (-2.,3.);
\draw [line width=.8pt] (-2.,3.)-- (-2.,4.);
\draw [line width=.8pt] (-2.,4.)-- (-1.1339745962155612,3.5);
\draw [line width=.8pt] (-1.1339745962155612,3.5)-- (-2.,3.);
\draw [line width=.8pt] (-2.,3.)-- (-1.1339745962155616,2.5);
\draw [line width=.8pt] (-1.1339745962155616,2.5)-- (-1.1339745962155612,3.5);
\draw [line width=.8pt] (-1.1339745962155612,3.5)-- (-0.2679491924311225,3.);
\draw [line width=.8pt] (-0.2679491924311225,3.)-- (-1.1339745962155616,2.5);
\draw [line width=.8pt,dotted] (-1.1339745962155616,2.5)-- (-0.26794919243112303,2.);
\draw [line width=.8pt,dotted] (-0.26794919243112303,2.)-- (-0.2679491924311225,3.);
\draw [line width=.8pt,dotted] (-0.2679491924311225,3.)-- (0.5980762113533165,2.5);
\draw [line width=.8pt] (0.5980762113533165,2.5)-- (-0.26794919243112303,2.);
\draw [line width=.8pt] (-0.26794919243112303,2.)-- (0.5980762113533152,1.5);
\draw [line width=.8pt] (0.5980762113533152,1.5)-- (0.5980762113533165,2.5);
\draw [line width=.8pt] (0.5980762113533165,2.5)-- (1.464101615137755,2.);
\draw [line width=.8pt] (1.464101615137755,2.)-- (0.5980762113533152,1.5);
\draw [line width=.8pt] (-2.,4.)-- (-2.866025403784439,3.5);
\draw [line width=.8pt,dotted] (-2.866025403784439,3.5)-- (-3.7320508075688785,3.);
\draw [line width=.8pt] (-3.7320508075688785,3.)-- (-4.598076211353318,2.5);
\draw [line width=.8pt] (-4.598076211353318,2.5)-- (-5.464101615137757,2.);
\draw [line width=.8pt] (-4.598076211353318,1.5)-- (-3.732050807568879,2.);
\draw [line width=.8pt,dotted] (-3.732050807568879,2.)-- (-2.8660254037844393,2.5);
\draw [line width=.8pt] (-2.8660254037844393,2.5)-- (-2.,3.);
\begin{scriptsize}
\draw [fill=black] (-2.,4.) circle (1.5pt);
\draw[color=black] (-2.06648828953202,4.259331085745072) node {$k+1$};
\draw [fill=black] (-1.1339745962155612,3.5) circle (1.5pt);
\draw[color=black] (-0.9407359844212153,3.7428094398707032) node {$k+2$};
\draw [fill=black] (-2.,3.) circle (1.5pt);
\draw[color=black] (-2.0267558552339913,2.6989231016718021) node {$k$};
\draw [fill=black] (-2.866025403784439,3.5) circle (1.5pt);
\draw[color=black] (-3.0597991469827295,3.689832860806665) node {$k-1$};
\draw [fill=black] (-2.8660254037844393,2.5) circle (1.5pt);
\draw[color=black] (-2.702207238300474,2.2919066737377372) node {$k-2$};
\draw [fill=black] (-1.1339745962155616,2.5) circle (1.5pt);
\draw[color=black] (-1.3248161826354898,2.2919066737377372) node {$k+3$};
\draw [fill=black] (-2.,3.) circle (1.5pt);
\draw [fill=black] (-0.2679491924311225,3.) circle (1.5pt);
\draw[color=black] (-0.14608729846064736,3.213043649230324) node {$k+4$};
\draw [fill=black] (-3.7320508075688785,3.) circle (1.5pt);
\draw[color=black] (-3.9339127015393545,3.2395319387623434) node {$5$};
\draw [fill=black] (-3.732050807568879,2.) circle (1.5pt);
\draw[color=black] (-3.6028090823891175,1.8488967383313488) node {$4$};
\draw [fill=black] (-0.26794919243112303,2.) circle (1.5pt);
\draw[color=black] (-0.4374584833128556,1.7488967383313488) node {$n-3$};
\draw [fill=black] (-4.598076211353318,2.5) circle (1.5pt);
\draw[color=black] (-4.78153796656396,2.7494985824199927) node {$3$};
\draw [fill=black] (-4.598076211353318,1.5) circle (1.5pt);
\draw[color=black] (-4.463678492179733,1.345619237222989) node {$2$};
\draw [fill=black] (-5.464101615137757,2.) circle (1.5pt);
\draw[color=black] (-5.655651521120585,2.1270237784175476) node {$1$};
\draw [fill=black] (0.5980762113533165,2.5) circle (1.5pt);
\draw[color=black] (0.7280262560959774,2.709766148121964) node {$n-2$};
\draw [fill=black] (0.5980762113533152,1.5) circle (1.5pt);
\draw[color=black] (0.4234109264777597,1.2721075267550078) node {$n-1$};
\draw [fill=black] (1.464101615137755,2.) circle (1.5pt);
\draw[color=black] (1.628628100184621,2.0475589098214906) node {$n$};
\end{scriptsize}
\end{tikzpicture}
\end{center}
\caption{A linear 2-tree with $n$ vertices and single bend at vertex $k$.}
\label{fig:bent}
\end{figure}
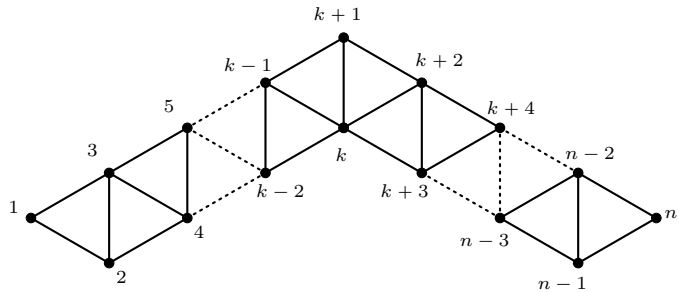
In essence a bent linear 2 tree differs from a straight linear 2 tree in that vertex $k$ has degree 5, vertex $k+1$ has degree $3$ and all other vertices have degrees as before.
Our main result is: 

\begin{main1} 
 Given a bent linear 2-tree with $n$ vertices, $m = n-2$ triangles, and one bend located at vertex $k$, the resistance distance between node 1 and node $n$ is
	 \begin{equation}r_{m,k}(1,n)=\frac{m+1}{5}+\frac{4F_{m+1}}{5L_{m+1}}+\frac{\sum_{j=3}^{k}\left[(-1)^jF_{m-2j+3}(F_{m+2}+F_{j-2}F_{m-j+1})\right]}{F_{2m+2}},\end{equation}\noindent where $F_p$ is the $p$th Fibonacci number and $L_q$ is the $q$th Lucas number. 
\end{main1}

In Section~\ref{sec:note} we explain our main tool, the $\Delta$--Y transform, and cite important Fibonacci and Lucas number identities that will be used in the paper. In Section 3 we show our main result.

\section{Transformations and Necessary Identities}\label{sec:note}

In this paper we will determine the effective resistance between the extreme vertices of a bent linear 2--tree by performing a sequence of $\Delta$--Y transformations, and using the standard rules for circuits in series and parallel (instead of using the Moore-Penrose pseudo inverse of the graph Laplacian).  A $\Delta$--Y transformation is a mathematical technique to convert resistors in a triangle ($\Delta$) formation to an equivalent system of three resistors in a ``Y'' format as illustrated in Figure~\ref{fig:dy}.  We formalize this transformation below. 
\begin{definition}[$\Delta-Y$ transformation]\label{def:dy}
Let $N_1, N_2, N_3$ be nodes and $R_A$, $R_B$ and $R_C$ be given resistances as shown in Figure~\ref{fig:dy}.  The equivalent circuit in the ``Y'' format as shown in Figure~\ref{fig:dy} has the following resistances:
\begin{align}
  R_1 &= \frac{R_BR_C}{R_A + R_B + R_C} \\
  R_2 &= \frac{R_AR_C}{R_A + R_B + R_C} \\
  R_3 &= \frac{R_AR_B}{R_A + R_B + R_C}
\end{align}
\end{definition}
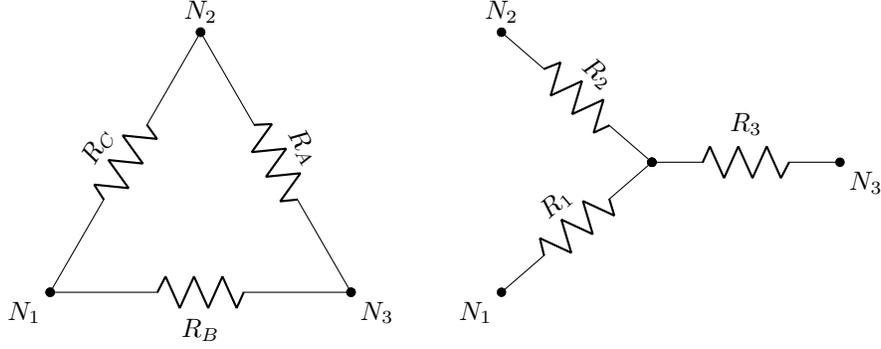
\begin{figure}
\begin{center}
\begin{circuitikz} \draw
 (0,0) to[R, *-*,a=$R_B$] (4,0);
 \draw (0,0) to[R, *-*,l=$R_C$] (2,3.46);
 \draw (2,3.46) to[R, *-*,l=$R_A$] (4,0)
 {[anchor=north east] (0,0) node {$N_1$} } {[anchor=north west] (4,0) node {$N_3$}} {[anchor=south]  (2,3.46) node {$N_2$}};
 \draw (6,0) to [R, *-*,a=$R_1$] (8,1.73);
 \draw (6,3.46) to [R, *-*,a=$R_2$] (8,1.73);
\draw (10.5,1.73) to [R, *-*,a=$R_3$] (8,1.73)
 {[anchor=north east] (6,0) node {$N_1$} } {[anchor=north west] (10.5,1.73) node {$N_3$}} {[anchor=south]  (6,3.46) node {$N_2$}};
; 
\end{circuitikz}
\end{center}
\caption{$\Delta$ and $Y$ circuits with vertices labeled as in Definition~\ref{def:dy}.}
\label{fig:dy}
\end{figure}

\subsection{Fibonacci and Lucas Identities}
Both Fibonacci and Lucas numbers are important to the proof of our main results.  By $F_n$ we denote the $n$th Fibonacci number and by $L_m$ we denote the $m$th Lucas number.
\begin{proposition} $F_{2m}=L_mF_m.$\end{proposition}
\begin{proof}
For the proof see~\cite{Vajda}.%
\end{proof}
\begin{proposition}\label{prop:splitsum} $F_{k+m}=F_mF_{k+1}+F_{m-1}F_k$.\end{proposition}
\begin{proof}
For the proof see~\cite{BQ}.%
\end{proof}
\begin{corollary} $F_{2m}=F_mF_{m+1}+F_{m-1}F_m.$\end{corollary}
\begin{corollary} $F_{2m-3}=F_{m-1}^2+F^2_{m-2}.$\end{corollary}
\begin{proposition}$F_{n + m} = F_{n + 1}F_{m + 1} - F_{n - 1}F_{m - 1}$\end{proposition}
\begin{proof}
For the proof see~\cite{BQ}.%
\end{proof}
\begin{corollary}$F_{2m} = F_{m+1}^2-F_{m - 1}^2$\end{corollary}
\begin{proposition}$F_{m+3}^2+F_m^2 = 2\left(F_{m+1}^2+F_{m +2}^2\right)$\end{proposition}
\begin{proof}
For the proof see~\cite{BQ}.%
\end{proof}
\begin{proposition}$3F_{m} = F_{m+2}+F_{m - 2}$\end{proposition}
\begin{proof}
For the proof see~\cite{BQ}.%
\end{proof}

\begin{proposition} $L_m=F_{m+1}+F_{m-1}.$\end{proposition}
\begin{proof}
For the proof see~\cite{Vajda}.
\end{proof}
\begin{proposition} $2F_{m+1}=F_m+L_{m}.$\end{proposition}
\begin{proof}
For the proof see~\cite{Vajda}.
\end{proof}
\begin{proposition} $L_{m+1}=2F_m+F_{m+1}.$\end{proposition}
\begin{proof}
For the proof see~\cite{BQ}.%
\end{proof}
\begin{proposition}\label{prop:Melham} For integers $a,b,c$, and $n$ the following relationship holds
\[F_{n+a+b+c}F_{n-a}F_{n-b}F_{n-c} - F_{n-a-b-c}F_{n+a}F_{n+b}F_{n+c}=(-1)^{n+a+b+c}F_{a+b}F_{a+c}F_{b+c}F_{2n}.\]
\end{proposition}
\begin{proof}
For the proof see \cite{Melham2011}.
\end{proof}
\begin{proposition}\label{prop:newprop2} $F_{n+i}F_{n+r}-F_nF_{n+i+r} = (-1)^nF_iF_r.$\end{proposition}
\begin{proof}
For the proof see~\cite{Vajda}.%
\end{proof}

\section{Maximal resistance distance in linear 2-trees with a bend at vertex $k$}

When determining the resistance distance between node $1$ and node $n$ we use a similar strategy to that detailed in~\cite{bef} and demonstrated in Figures ~\ref{fig:2treepost} and~\ref{fig:2treepostnext}
\begin{itemize}
\item First perform the $\Delta$--Y transform on the leftmost triangle (defined by the vertices $1$, $2$, and $3$).  This results in a new graph node $*$ as shown in Figure~\ref{fig:2treepost}.  
\item Next, sum the weight between vertices $2$ and $*$ with the weight between vertices $2$ and $4$, delete vertex 2 and rename vertex * as vertex 2 as shown in Figure~\ref{fig:2treepostnext}.  
\item Perform a $\Delta$-Y transform on the new left-most triangle.  
\end{itemize}
We continue in this manner until all the $k-2$nd triangle has been removed.  We then repeat this process starting with the rightmost triangle (defined by the vertices $n$, $n-1$ and $n-2$) until the $n-k-2$nd triangle is removed and we are left with two triangles and two long tail. (See Figure~\ref{fig:amanda3}).  We then perform one final $\Delta$--Y transformation on the $n-k-1$st triangle leaving us with a pair of parallel edges and a two long tails, as in Figure~\ref{fig:amanda4}.
\begin{remark}
Notice that the $j$th $\Delta$--Y transformation transforms a triangle with node labels $j,j +1$,  and $ j+2$ into a $Y$ with nodes labeled $j, j +1, j+2$, and $\ast$. We name the edges so that $R_A^j = r(j,j+1)$, $R_B^j = r(j,j+2)$, and $R_C^j = r(j+1,j+2)$.  Thus, in the subsequent, equivalent network, $t_j = R_3^j = r( \ast, j)$, $s_j = R_2^j = r( \ast, j+1)$, and $b_j = R_1^j = r( \ast, j+2)$.  
We call $t_k$ the tail resistance because after a sequence of $\Delta$--Y transforms several resistors are left in a tail.  This resistance will never be involved in another $\Delta$-Y transformation. Notice that when performing the merge step described in the second bullet above, $s_k$ will always be on the edge that is merged. It is easy to verify that $R_C^j=1$  for every $j$. 
\end{remark}
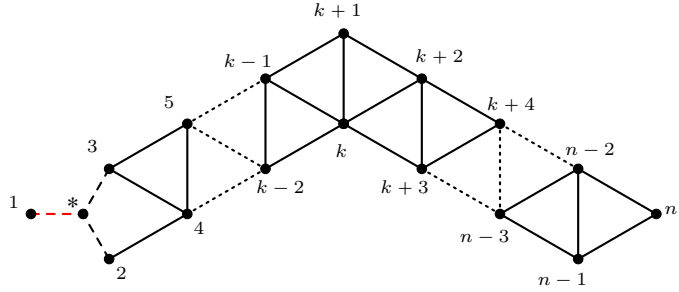
\begin{figure}
\begin{center}
\begin{tikzpicture}[line cap=round,line join=round,>=triangle 45,x=1.0cm,y=1.0cm,scale = 1.2]
\draw [line width=.8pt,dashed,red] (-5.464101615137757,2.)-- (-4.88675134594813,2.);
\draw [line width=.8pt,dashed] (-4.598076211353318,1.5)-- (-4.88675134594813,2.);
\draw [line width=.8pt,dashed] (-4.598076211353318,2.5)-- (-4.88675134594813,2.);
\draw [line width=.8pt] (-4.598076211353318,2.5)-- (-3.732050807568879,2.);
\draw [line width=.8pt] (-3.732050807568879,2.)-- (-3.7320508075688785,3.);
\draw [line width=.8pt,dotted] (-3.7320508075688785,3.)-- (-2.8660254037844393,2.5);
\draw [line width=.8pt] (-2.8660254037844393,2.5)-- (-2.866025403784439,3.5);
\draw [line width=.8pt] (-2.866025403784439,3.5)-- (-2.,3.);
\draw [line width=.8pt] (-2.,3.)-- (-2.,4.);
\draw [line width=.8pt] (-2.,4.)-- (-1.1339745962155612,3.5);
\draw [line width=.8pt] (-1.1339745962155612,3.5)-- (-2.,3.);
\draw [line width=.8pt] (-2.,3.)-- (-1.1339745962155616,2.5);
\draw [line width=.8pt] (-1.1339745962155616,2.5)-- (-1.1339745962155612,3.5);
\draw [line width=.8pt] (-1.1339745962155612,3.5)-- (-0.2679491924311225,3.);
\draw [line width=.8pt] (-0.2679491924311225,3.)-- (-1.1339745962155616,2.5);
\draw [line width=.8pt,dotted] (-1.1339745962155616,2.5)-- (-0.26794919243112303,2.);
\draw [line width=.8pt,dotted] (-0.26794919243112303,2.)-- (-0.2679491924311225,3.);
\draw [line width=.8pt,dotted] (-0.2679491924311225,3.)-- (0.5980762113533165,2.5);
\draw [line width=.8pt] (0.5980762113533165,2.5)-- (-0.26794919243112303,2.);
\draw [line width=.8pt] (-0.26794919243112303,2.)-- (0.5980762113533152,1.5);
\draw [line width=.8pt] (0.5980762113533152,1.5)-- (0.5980762113533165,2.5);
\draw [line width=.8pt] (0.5980762113533165,2.5)-- (1.464101615137755,2.);
\draw [line width=.8pt] (1.464101615137755,2.)-- (0.5980762113533152,1.5);
\draw [line width=.8pt] (-2.,4.)-- (-2.866025403784439,3.5);
\draw [line width=.8pt,dotted] (-2.866025403784439,3.5)-- (-3.7320508075688785,3.);
\draw [line width=.8pt] (-3.7320508075688785,3.)-- (-4.598076211353318,2.5);
\draw [line width=.8pt] (-4.598076211353318,1.5)-- (-3.732050807568879,2.);
\draw [line width=.8pt,dotted] (-3.732050807568879,2.)-- (-2.8660254037844393,2.5);
\draw [line width=.8pt] (-2.8660254037844393,2.5)-- (-2.,3.);
\begin{scriptsize}
\draw [fill=black] (-2.,4.) circle (1.5pt);
\draw[color=black] (-2.06648828953202,4.259331085745072) node {$k+1$};
\draw [fill=black] (-1.1339745962155612,3.5) circle (1.5pt);
\draw[color=black] (-0.9407359844212153,3.7428094398707032) node {$k+2$};
\draw [fill=black] (-2.,3.) circle (1.5pt);
\draw[color=black] (-2.0267558552339913,2.6989231016718021) node {$k$};
\draw [fill=black] (-2.866025403784439,3.5) circle (1.5pt);
\draw[color=black] (-3.0597991469827295,3.689832860806665) node {$k-1$};
\draw [fill=black] (-2.8660254037844393,2.5) circle (1.5pt);
\draw[color=black] (-2.702207238300474,2.2919066737377372) node {$k-2$};
\draw [fill=black] (-1.1339745962155616,2.5) circle (1.5pt);
\draw[color=black] (-1.3248161826354898,2.2919066737377372) node {$k+3$};
\draw [fill=black] (-2.,3.) circle (1.5pt);
\draw [fill=black] (-0.2679491924311225,3.) circle (1.5pt);
\draw[color=black] (-0.14608729846064736,3.213043649230324) node {$k+4$};
\draw [fill=black] (-3.7320508075688785,3.) circle (1.5pt);
\draw[color=black] (-3.9339127015393545,3.2395319387623434) node {$5$};
\draw [fill=black] (-3.732050807568879,2.) circle (1.5pt);
\draw[color=black] (-3.6028090823891175,1.8488967383313488) node {$4$};
\draw [fill=black] (-0.26794919243112303,2.) circle (1.5pt);
\draw[color=black] (-0.4374584833128556,1.7488967383313488) node {$n-3$};
\draw [fill=black] (-4.598076211353318,2.5) circle (1.5pt);
\draw[color=black] (-4.78153796656396,2.7494985824199927) node {$3$};
\draw [fill=black] (-4.598076211353318,1.5) circle (1.5pt);
\draw[color=black] (-4.463678492179733,1.345619237222989) node {$2$};
\draw [fill=black] (-5.464101615137757,2.) circle (1.5pt);
\draw[color=black] (-5.655651521120585,2.1270237784175476) node {$1$};
\draw [fill=black] (-4.88675134594813,2.) circle (1.5pt);
\draw[color=black] (-5.0,2.1270237784175476) node {$*$};
\draw [fill=black] (0.5980762113533165,2.5) circle (1.5pt);
\draw[color=black] (0.7280262560959774,2.709766148121964) node {$n-2$};
\draw [fill=black] (0.5980762113533152,1.5) circle (1.5pt);
\draw[color=black] (0.4234109264777597,1.2721075267550078) node {$n-1$};
\draw [fill=black] (1.464101615137755,2.) circle (1.5pt);
\draw[color=black] (1.628628100184621,2.0475589098214906) node {$n$};
\end{scriptsize}
\end{tikzpicture}

%
%
\end{center}
\caption{A bent 2-tree after the first $\Delta$--Y transformation on the left.  The dashed edges are the edges with new weights after the transformation.  The dashed red edge (from node 1 to node $\ast$) is the ``tail'' of the transformation.}
\label{fig:2treepost}
\end{figure}

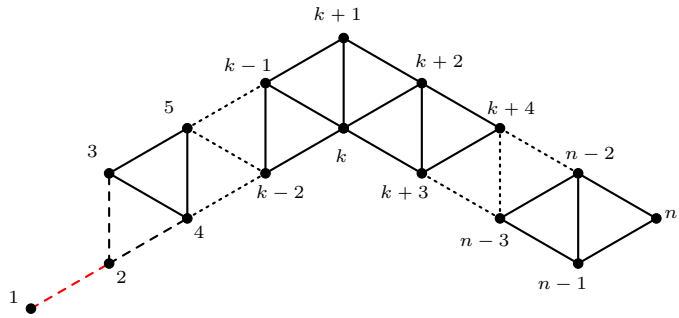
\begin{figure}
\begin{center}
\begin{tikzpicture}[line cap=round,line join=round,>=triangle 45,x=1.0cm,y=1.0cm,scale = 1.2]
\draw [line width=.8pt,dashed] (-4.598076211353318,1.5)-- (-4.598076211353318,2.5);
\draw [line width=.8pt,dashed,red] (-5.464101615137757,1.)-- (-4.598076211353318,1.5);
\draw [line width=.8pt] (-4.598076211353318,2.5)-- (-3.732050807568879,2.);
\draw [line width=.8pt] (-3.732050807568879,2.)-- (-3.7320508075688785,3.);
\draw [line width=.8pt,dotted] (-3.7320508075688785,3.)-- (-2.8660254037844393,2.5);
\draw [line width=.8pt] (-2.8660254037844393,2.5)-- (-2.866025403784439,3.5);
\draw [line width=.8pt] (-2.866025403784439,3.5)-- (-2.,3.);
\draw [line width=.8pt] (-2.,3.)-- (-2.,4.);
\draw [line width=.8pt] (-2.,4.)-- (-1.1339745962155612,3.5);
\draw [line width=.8pt] (-1.1339745962155612,3.5)-- (-2.,3.);
\draw [line width=.8pt] (-2.,3.)-- (-1.1339745962155616,2.5);
\draw [line width=.8pt] (-1.1339745962155616,2.5)-- (-1.1339745962155612,3.5);
\draw [line width=.8pt] (-1.1339745962155612,3.5)-- (-0.2679491924311225,3.);
\draw [line width=.8pt] (-0.2679491924311225,3.)-- (-1.1339745962155616,2.5);
\draw [line width=.8pt,dotted] (-1.1339745962155616,2.5)-- (-0.26794919243112303,2.);
\draw [line width=.8pt,dotted] (-0.26794919243112303,2.)-- (-0.2679491924311225,3.);
\draw [line width=.8pt,dotted] (-0.2679491924311225,3.)-- (0.5980762113533165,2.5);
\draw [line width=.8pt] (0.5980762113533165,2.5)-- (-0.26794919243112303,2.);
\draw [line width=.8pt] (-0.26794919243112303,2.)-- (0.5980762113533152,1.5);
\draw [line width=.8pt] (0.5980762113533152,1.5)-- (0.5980762113533165,2.5);
\draw [line width=.8pt] (0.5980762113533165,2.5)-- (1.464101615137755,2.);
\draw [line width=.8pt] (1.464101615137755,2.)-- (0.5980762113533152,1.5);
\draw [line width=.8pt] (-2.,4.)-- (-2.866025403784439,3.5);
\draw [line width=.8pt,dotted] (-2.866025403784439,3.5)-- (-3.7320508075688785,3.);
\draw [line width=.8pt] (-3.7320508075688785,3.)-- (-4.598076211353318,2.5);
\draw [line width=.8pt,dashed] (-4.598076211353318,1.5)-- (-3.732050807568879,2.);
\draw [line width=.8pt,dotted] (-3.732050807568879,2.)-- (-2.8660254037844393,2.5);
\draw [line width=.8pt] (-2.8660254037844393,2.5)-- (-2.,3.);
\begin{scriptsize}
\draw [fill=black] (-2.,4.) circle (1.5pt);
\draw[color=black] (-2.06648828953202,4.259331085745072) node {$k+1$};
\draw [fill=black] (-1.1339745962155612,3.5) circle (1.5pt);
\draw[color=black] (-0.9407359844212153,3.7428094398707032) node {$k+2$};
\draw [fill=black] (-2.,3.) circle (1.5pt);
\draw[color=black] (-2.0267558552339913,2.6989231016718021) node {$k$};
\draw [fill=black] (-2.866025403784439,3.5) circle (1.5pt);
\draw[color=black] (-3.0597991469827295,3.689832860806665) node {$k-1$};
\draw [fill=black] (-2.8660254037844393,2.5) circle (1.5pt);
\draw[color=black] (-2.702207238300474,2.2919066737377372) node {$k-2$};
\draw [fill=black] (-1.1339745962155616,2.5) circle (1.5pt);
\draw[color=black] (-1.3248161826354898,2.2919066737377372) node {$k+3$};
\draw [fill=black] (-2.,3.) circle (1.5pt);
\draw [fill=black] (-0.2679491924311225,3.) circle (1.5pt);
\draw[color=black] (-0.14608729846064736,3.213043649230324) node {$k+4$};
\draw [fill=black] (-3.7320508075688785,3.) circle (1.5pt);
\draw[color=black] (-3.9339127015393545,3.2395319387623434) node {$5$};
\draw [fill=black] (-3.732050807568879,2.) circle (1.5pt);
\draw[color=black] (-3.6028090823891175,1.8488967383313488) node {$4$};
\draw [fill=black] (-0.26794919243112303,2.) circle (1.5pt);
\draw[color=black] (-0.4374584833128556,1.7488967383313488) node {$n-3$};
\draw [fill=black] (-4.598076211353318,2.5) circle (1.5pt);
\draw[color=black] (-4.78153796656396,2.7494985824199927) node {$3$};
\draw [fill=black] (-4.598076211353318,1.5) circle (1.5pt);
\draw[color=black] (-4.463678492179733,1.345619237222989) node {$2$};
\draw [fill=black] (-5.464101615137757,1.) circle (1.5pt);
\draw[color=black] (-5.655651521120585,1.1270237784175476) node {$1$};
\draw [fill=black] (0.5980762113533165,2.5) circle (1.5pt);
\draw[color=black] (0.7280262560959774,2.709766148121964) node {$n-2$};
\draw [fill=black] (0.5980762113533152,1.5) circle (1.5pt);
\draw[color=black] (0.4234109264777597,1.2721075267550078) node {$n-1$};
\draw [fill=black] (1.464101615137755,2.) circle (1.5pt);
\draw[color=black] (1.628628100184621,2.0475589098214906) node {$n$};
\end{scriptsize}
\end{tikzpicture}%
%
%
\end{center}
\caption{A bent linear 2-tree after the first $\Delta$--Y transformation on the left.  Once two edges are merged, a vertex is removed, and a vertex is renamed.  The dashed edges are the edges with new weights after the transformation.  The dashed red edge (from node 1 to node 2) is the ``tail'' of the transformation.}
\label{fig:2treepostnext}
\end{figure}

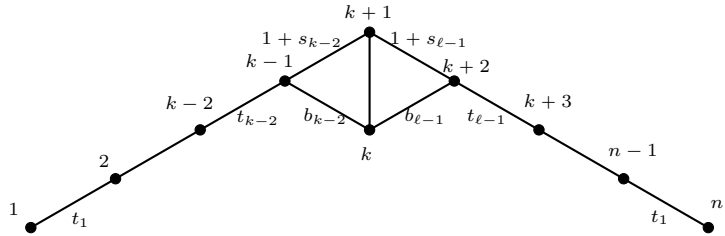
\begin{figure}
\begin{center}

\begin{tikzpicture}[line cap=round,line join=round,>=triangle 45,x=1.0cm,y=1.0cm,scale = 1.3]

\draw [line width=.8pt] (-0.464101615137757,1.)-- (0.40192378864668243,1.5);
\draw [line width=.8pt] (0.40192378864668243,1.5)-- (1.2679491924311215,2.);
\draw [line width=.8pt] (1.2679491924311215,2.)-- (2.1339745962155607,2.5);
\draw [line width=.8pt] (3.,2.)-- (2.1339745962155607,2.5);
\draw [line width=.8pt] (2.1339745962155607,2.5)-- (3.,3.);
\draw [line width=.8pt] (3.,3.)-- (3.,2.);
\draw [line width=.8pt] (3.,2.)-- (3.8660254037844393,2.5);
\draw [line width=.8pt] (3.,3.)-- (3.8660254037844393,2.5);
\draw [line width=.8pt] (3.8660254037844393,2.5)-- (4.7320508075688785,2.);
\draw [line width=.8pt] (4.7320508075688785,2.)-- (5.598076211353318,1.5);
\draw [line width=.8pt] (5.598076211353319,1.5)-- (6.464101615137759,1.);
\begin{scriptsize}
\draw [fill=black] (3.,2.) circle (1.5pt);
\draw[color=black] (2.9674492066303038,1.757563732424974) node {$k$};
\draw [fill=black] (3.,3.) circle (1.5pt);
\draw[color=black] (2.9872376055767735,3.203174860784933) node {$k+1$};
\draw [fill=black] (2.1339745962155607,2.5) circle (1.5pt);
\draw[color=black] (1.9780292593068036,2.699311291337303) node {$k-1$};
\draw [fill=black] (3.8660254037844393,2.5) circle (1.5pt);
\draw[color=black] (3.9865517523735083,2.65194170918640687) node {$k+2$};
\draw[color=black] (0.04371326228936186,1.0798110685083771) node {$t_1$};
\draw[color=black] (1.8555155716857813,2.10890194147783464) node {$t_{k-2}$};
\draw[color=black] (2.546945729017816,2.1285962126712864) node {$b_{k-2}$};
\draw[color=black] (2.307897933758701,2.9110661173163911) node {$1 + s_{k-2}$};
\draw[color=black] (3.566048274761021,2.1285962126712864) node {$b_{\ell-1}$};
\draw[color=black] (3.6066830779216107,2.911295976324501) node {$1+s_{\ell-1}$};
\draw[color=black] (4.19927704104806,2.1187020131980515) node {$t_{\ell-1}$};
\draw [fill=black] (1.2679491924311217,2.) circle (1.5pt);
\draw[color=black] (1.157233905135359,2.274178115568494) node {$k-2$};
\draw [fill=black] (0.4019237886466823,1.5) circle (1.5pt);
\draw[color=black] (0.2861211493836193,1.6794681419067443) node {$2$};
\draw [fill=black] (-0.464101615137757,1.) circle (1.5pt);
\draw[color=black] (-0.638986501363853,1.1897052679816122) node {$1$};
\draw [fill=black] (4.732050807568879,2.) circle (1.5pt);
\draw[color=black] (4.822611607861866,2.2890194147783464) node {$k+3$};
\draw [fill=black] (5.598076211353319,1.5) circle (1.5pt);
\draw[color=black] (5.693301161506545,1.7844152416433618) node {$n-1$};
\draw [fill=black] (6.464101615137759,1.) circle (1.5pt);
\draw[color=black] (6.5540965156779905,1.2397052679816122) node {$n$};
\draw[color=black] (5.970338746757125,1.10699168690351422) node {$t_1$};
\end{scriptsize}
\end{tikzpicture}
\end{center}
\caption{A bent linear 2-tree after $k-2$ transformations on the left and $\ell-1=n-k-2$ transformations on the right.  }
\label{fig:amanda3}
\end{figure}

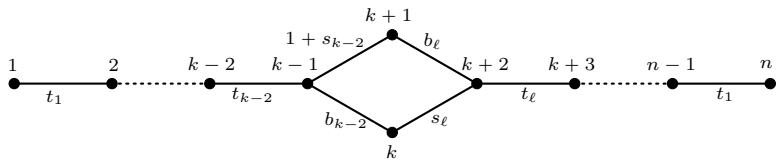
\begin{figure}
\begin{center}

\begin{tikzpicture}[line cap=round,line join=round,>=triangle 45,x=1.0cm,y=1.0cm,scale = 1.3]

\draw [line width=.8pt] (0.1339745962155613,2.5)-- (1.,3.);
\draw [line width=.8pt] (1.,3.)-- (1.8660254037844388,2.5);
\draw [line width=.8pt] (1.8660254037844388,2.5)-- (1.,2.);
\draw [line width=.8pt] (1.,2.)-- (0.1339745962155613,2.5);
\draw [line width=.8pt] (1.8660254037844388,2.5)-- (2.866025403784439,2.5);
\draw [line width=.8pt,dotted] (2.866025403784439,2.5)-- (3.866025403784439,2.5);
\draw [line width=.8pt] (3.866025403784439,2.5)-- (4.866025403784439,2.5);
\draw [line width=.8pt] (0.1339745962155613,2.5)-- (-0.8660254037844387,2.5);
\draw [line width=.8pt,dotted] (-0.8660254037844387,2.5)-- (-1.866025374954007,2.5002401267639827);
\draw [line width=.8pt] (-1.866025374954007,2.5002401267639827)-- (-2.866025374954007,2.5002401267639827);
\begin{scriptsize}
\draw [fill=black] (1.,2.) circle (1.5pt);
\draw[color=black] (0.9943675155481322,1.8043542220683444) node {$k$};
\draw [fill=black] (1.,3.) circle (1.5pt);
\draw[color=black] (0.9643675155481322,3.2009587792214767) node {$k+1$};
\draw [fill=black] (0.1339745962155613,2.5) circle (1.5pt);
\draw[color=black] (0.006497989491263595,2.692157170814489) node {$k-1$};
\draw [fill=black] (1.8660254037844388,2.5) circle (1.5pt);
\draw[color=black] (1.9521704642942779,2.6971238821591275) node {$k+2$};
\draw[color=black] (0.5303328865536137,2.118588582995031) node {$b_{k-2}$};
\draw[color=black] (1.4133688558872892,2.936657840984067) node {$b_\ell$};
\draw[color=black] (1.4882024126104823,2.118455428373585) node {$s_\ell$};
\draw[color=black] (0.315366175208975,2.936591263673344) node {$1+ s_{k-2}$};
\draw [fill=black] (2.866025403784439,2.5) circle (1.5pt);
\draw[color=black] (2.8352064336279534,2.6971238821591275) node {$k+3$};
\draw[color=black] (2.401171804633435,2.387989387198525) node {$t_\ell$};
\draw [fill=black] (3.866025403784439,2.5) circle (1.5pt);
\draw[color=black] (3.8529428050633765,2.692090593503766) node {$ n-1$};
\draw [fill=black] (4.866025403784439,2.5) circle (1.5pt);
\draw[color=black] (4.825779042464883,2.6971238821591275) node {$n$};
\draw[color=black] (4.406711124815003,2.387989387198525) node {$t_1$};
\draw [fill=black] (-0.8660254037844387,2.5) circle (1.5pt);
\draw[color=black] (-0.846604557153135,2.692090593503766) node {$k-2$};
\draw[color=black] (-0.42753663950325493,2.3730226758538864) node {$t_{k-2}$};
\draw [fill=black] (-1.866025374954007,2.5002401267639827) circle (1.5pt);
\draw[color=black] (-1.8493742172439194,2.692024016193043) node {$2$};
\draw [fill=black] (-2.866025374954007,2.5002401267639827) circle (1.5pt);
\draw[color=black] (-2.882077300023981,2.692090593503766) node {$1$};
\draw[color=black] (-2.4480426710294623,2.358055964509248) node {$t_1$};
\end{scriptsize}
\end{tikzpicture}

\end{center}
\caption{A bent linear 2-tree after the $k-2$ transformations on the left and $\ell=n-k-1$ transformations on the right.  This leaves a set of parallel edges and two long tails.}
\label{fig:amanda4}
\end{figure}

	\begin{theorem}\label{thm:effResBend}
	Let $G$ be a bent linear 2-tree with $n$ vertices, $m=n-2$ triangles and a single bend at vertex $k$. Letting $\ell=m-k+1$, then the effective resistance between nodes $1$ and $n$ is given by
	\begin{equation}\label{eq:embent}r_{m,k}(1,n)=\frac{(F_{2\ell+2}F_{k-1}^2+F_{2k-2}F^2_{\ell})(F_{2\ell+2}F^2_{k-2}+F_{2k-2}F^2_{\ell+1}+F_{2k-2}F_{2\ell+2})}{F_{2k-2}F_{2\ell+2}F_{2m+2}}+\sum_{i=1}^{k-2}\frac{F_iF_{i+1}}{L_iL_{i+1}}+\sum_{i=1}^{\ell}\frac{F_iF_{i+1}}{L_iL_{i+1}}.\end{equation}
	\end{theorem}
	\begin{proof}
	For ease of notation, let $p=k-2$.  Starting at node 1, perform $p\;$ $\Delta-Y$ transformations.  Similarly, starting at node $n$ perform $\ell\;$ $\Delta$-Y transformations.  This results in the configuration shown Figure~\ref{fig:amanda4}.
Thus 
\begin{align*}
r_{m,k}(1,n) &= \left(\frac{1}{{b_p}+{s_{\ell}}}+\frac{1}{{s_p}+{b_{\ell}}+1}\right)^{-1}+\sum_{i=1}^{k-2}\frac{F_iF_{i+1}}{L_iL_{i+1}}+\sum_{i=1}^{\ell}\frac{F_iF_{i+1}}{L_iL_{i+1}}\\
&=\left(\frac{1}{\frac{F_{p+1}}{L_{p+1}}+\frac{F^2_{\ell}}{F_{2\ell+2}}}+\frac{1}{\frac{F^2_{p}}{F_{2p+2}}+\frac{F_{\ell+1}}{L_{\ell+1}} +1}\right)^{-1}+\sum_{i=1}^{k-2}\frac{F_iF_{i+1}}{L_iL_{i+1}}+\sum_{i=1}^{\ell}\frac{F_iF_{i+1}}{L_iL_{i+1}}\quad\text{(By~\cite[Lemma 28]{bef})}\\	
&=\left(\frac{F_{2p+2}F_{2\ell+2}}{F_{2\ell+2}F^2_{p+1}+F_{2p+2}F^2_{\ell}}+\frac{F_{2p+2}F_{2\ell+2}}{F_{2\ell+2}F^2_p+F_{2p+2}F^2_{\ell+1}+F_{2p+2}F_{2\ell+2}}\right)^{-1}+\sum_{i=1}^{k-2}\frac{F_iF_{i+1}}{L_iL_{i+1}}+\sum_{i=1}^{\ell}\frac{F_iF_{i+1}}{L_iL_{i+1}}\\	
&=\frac{(F_{2\ell+2}F_{k-1}^2+F_{2k-2}F^2_{\ell})(F_{2\ell+2}F^2_{k-2}+F_{2k-2}F^2_{\ell+1}+F_{2k-2}F_{2\ell+2})}{F_{2p+2}F_{2\ell+2}(F_{2\ell+2}(F^2_{p+1}+F^2_{p})+F_{2p+2}(F^2_{\ell+1} + F^2_{\ell})+F_{2\ell+2}F_{2p+2})}+\sum_{i=1}^{k-2}\frac{F_iF_{i+1}}{L_iL_{i+1}}+\sum_{i=1}^{\ell}\frac{F_iF_{i+1}}{L_iL_{i+1}}\\	
&=\frac{(F_{2\ell+2}F_{k-1}^2+F_{2k-2}F^2_{\ell})(F_{2\ell+2}F^2_{k-2}+F_{2k-2}F^2_{\ell+1}+F_{2k-2}F_{2\ell+2})}{F_{2p+2}F_{2\ell+2}(F_{2\ell+2}F_{2p+1}+F_{2p+2}F_{2\ell+1}+F_{2\ell+2}F_{2p+2})}+\sum_{i=1}^{k-2}\frac{F_iF_{i+1}}{L_iL_{i+1}}+\sum_{i=1}^{\ell}\frac{F_iF_{i+1}}{L_iL_{i+1}}\\	
&=\frac{(F_{2\ell+2}F_{k-1}^2+F_{2k-2}F^2_{\ell})(F_{2\ell+2}F^2_{k-2}+F_{2k-2}F^2_{\ell+1}+F_{2k-2}F_{2\ell+2})}{F_{2k-2}F_{2\ell+2}F_{2p+2\ell+4}}+\sum_{i=1}^{k-2}\frac{F_iF_{i+1}}{L_iL_{i+1}}+\sum_{i=1}^{\ell}\frac{F_iF_{i+1}}{L_iL_{i+1}}\\	
&=\frac{(F_{2\ell+2}F_{k-1}^2+F_{2k-2}F^2_{\ell})(F_{2\ell+2}F^2_{k-2}+F_{2k-2}F^2_{\ell+1}+F_{2k-2}F_{2\ell+2})}{F_{2k-2}F_{2\ell+2}F_{2m+2}}+\sum_{i=1}^{k-2}\frac{F_iF_{i+1}}{L_iL_{i+1}}+\sum_{i=1}^{\ell}\frac{F_iF_{i+1}}{L_iL_{i+1}}.\end{align*}
	\end{proof}

\begin{lemma}For integers $m$ and $k$ the relationship \[F_{2k}F_{m-k-1}F_{m-k+2}=F_{m+1}F_{m-2k}F_{k+1}F_{k-2}+F_{m-2k+1}F_mF_{k-1}F_{k+2}\] holds.
\end{lemma}
\begin{proof} Applying Proposition~\ref{prop:Melham} with $n=k$, $a =m-k$, $b=-1$ and $c=2$ yields
\[(-1)^{m+1}F_{2k}F_{m-k-1}F_{m-k+2}=F_{m+1}F_{2k-m}F_{k+1}F_{k-2}-F_{2k-m-1}F_mF_{k-1}F_{k+2}.\]
Since $F_{-r}=(-1)^{r+1}F_r$, this is equivalent to
\[\begin{aligned}F_{2k}F_{m-k-1}F_{m-k+2}&=(-1)^{m+1}(-1)^{m-2k+1}F_{m+1}F_{m-2k}F_{k+1}F_{k-2}\\
&\quad\quad-(-1)^{m+1}(-1)^{m-2k+2}F_{m-2k+1}F_mF_{k-1}F_{k+2}\\
&=F_{m+1}F_{m-2k}F_{k+1}F_{k-2}+F_{m-2k+1}F_mF_{k-1}F_{k+2}.\end{aligned}\]
\end{proof}
\begin{lemma}\label{lem34}
For integers $m$ and $k$ the relationship
\[F_{k+1}F_m-F_{k-1}F_{m+2}=(-1)^{k-1}F_{m-k+1}\]
holds.
\end{lemma}
\begin{proof} Applying proposition \ref{prop:newprop2} with $n=k-1$, $i=2$, and $r=m-k+1$ yields the desired equality.
\end{proof}
\begin{lemma}\label{lem35}
For integers $m$ and $k$ the relationship \[F_{2m-2k+2}F_{k+1}F_{k-2}-F_{2k}F_{m-k-1}F_{m-k+2}=(-1)^{k+1}F_{m-2k+1}(F_{m+2}+F_{k-1}F_{m-k})\] holds.
\end{lemma}
\begin{proof}
Using Lemma~\ref{lem34}
\[\begin{aligned}F_{2m-2k+2}F_{k+1}F_{k-2}&-F_{2k}F_{m-k-1}F_{m-k+2}\\
&=F_{k+1}F_{k-2}(F_{m-2k+1}F_{m+2}+F_{m-2k}F_{m+1})- F_{2k}F_{m-k-1}F_{m-k+2}\\
&=F_{k+1}F_{k-2}(F_{m-2k+1}F_{m+2}+F_{m-2k}F_{m+1})\\
&\quad\quad- F_{m+1}F_{m-2k}F_{k+1}F_{k-2}-F_{m-2k+1}F_mF_{k-1}F_{k+2}\\
&=F_{k+1}F_{k-2}F_{m-2k+1}F_{m+2}- F_{m-2k+1}F_mF_{k-1}F_{k+2}\\
&=((-1)^{k+1}+F_kF_{k-1})F_{m-2k+1}F_{m+2}-((-1)^k+F_kF_{k+1})F_{m-2k+1}F_m\\
&=(-1)^{k+1}F_{m-2k+1}(F_{m+2}+F_m)+F_kF_{m-2k+1}(F_{k-1}F_{m+2}-F_{k+1}F_m)\\
&=(-1)^{k+1}F_{m-2k+1}(F_{m+2}+F_{m-k+k}-F_{k}F_{m-k+1})\\
&=(-1)^{k+1}F_{m-2k+1}(F_{m+2}+F_{k-1}F_{m-k}).\end{aligned}\]
\end{proof}

\begin{theorem}\label{thm:bend} 
Given a bent linear 2-tree with $n$ vertices, $m = n-2$ triangles, and one bend located at vertex $k \in\{3,4,\ldots,n-3\}$, the resistance distance between node 1 and node $n$ is
	 \begin{equation}\label{eq:univconf}r_{m,k}(1,n)=\frac{m+1}{5}+\frac{4F_{m+1}}{5L_{m+1}}+\frac{\sum_{j=3}^{k}\left[(-1)^jF_{m-2j+3}(F_{m+2}+F_{j-2}F_{m-j+1})\right]}{F_{2m+2}},\end{equation}\noindent where $F_p$ is the $p$th Fibonacci number and $L_q$ is the $q$th Lucas number. 
\	 \end{theorem}
	 \begin{proof}
	 Fix $n$ and $3\leq k \leq n-3$, and we induct on $k$.

	 Base Case ($k=3)$: From Theorem~\ref{thm:effResBend} 
	 \[\begin{split}r_{m,3}(1,n)&=\frac{(F_{2m-2}+F_{4}F^2_{m-2})(F_{2m-2}(1+F_4)+F_4F^2_{m-1})}{F_4F_{2\ell+2}F_{2m+2}}+\frac{1}{3}+\sum_{i=1}^{m-2}\frac{F_iF_{i+1}}{L_iL_{i+1}}\\
	 &=\frac{(F_{2m-2}+3F^2_{m-2})(4F_{2m-2}+3F^2_{m-1})}{3F_{2\ell+2}F_{2m+2}}+\frac{1}{3}+\sum_{i=1}^{m-2}\frac{F_iF_{i+1}}{L_iL_{i+1}}.\end{split}\]  By Lemma~\ref{base} in the Appendix this quantity is equal to
\[\frac{m+1}{5}+\frac{4F_{m+1}}{5L_{m+1}} -\frac{F_{m-3}(F_{m+2}+F_{m-2})}{F_{2m+2}} = \frac{m+1}{5}+\frac{4F_{m+1}}{5L_{m+1}}+\frac{\sum_{j=3}^{3}\left[(-1)^jF_{m-2j+3}(F_{m+2}+F_{j-2}F_{m-j+1})\right]}{F_{2m+2}}.\]	 
	 	We now assume that~\eqref{eq:univconf} holds for $k$ and show it also holds for $k+1$. Letting $\ell=m-k+1$ and $p=k-2$, from~\eqref{eq:embent} 
		\[\begin{aligned}&r_{m,k+1}(1,n)-r_{m,k}(1,n) =\frac{F_{k-1}F_k}{L_{k-1}L_k}-\frac{F_{\ell}F_{\ell+1}}{L_{\ell}L_{\ell+1}}+\frac{(F_{2\ell}F_{p+2}^2+F_{2p+4}F^2_{\ell-1})(F_{2\ell}F^2_{p+1}+F_{2p+4}F^2_{\ell}+F_{2p+4}F_{2\ell})}{F_{2p+4}F_{2\ell}F_{2m+2}}\\
		&\quad\quad\quad\quad-\frac{(F_{2\ell+2}F_{p+1}^2+F_{2p+2}F^2_{\ell})(F_{2\ell+2}F^2_{p}+F_{2p+2}F^2_{\ell+1}+F_{2p+2}F_{2\ell+2})}{F_{2p+2}F_{2\ell+2}F_{2m+2}}\\
		&=\frac{F_{p+1}^2F_{p+2}^2F_{2\ell}F_{2\ell+2}F_{2p+2\ell+4}-F_{\ell}^2F_{\ell+1}^2F_{2p+2}F_{2p+4}F_{2m+2}}{F_{2p+2}F_{2p+4}F_{2\ell}F_{2\ell+2}F_{2m+2}}\\
		&\quad+\frac{F_{2\ell+2}F_{2p+2}(F_{2\ell}^2F^2_{p+1}F^2_{p+2}+F_{2\ell}F^2_{p+2}F_{2p+4}F^2_{\ell}) }{F_{2p+2}F_{2p+4}F_{2\ell}F_{2\ell+2}F_{2m+2}}\\
		&\quad+ \frac{F_{2\ell+2}F_{2p+2}(F^2_{2\ell}F^2_{p+2}F_{2p+4}+F_{2p+4}F^2_{\ell-1}F_{2\ell}F^2_{p+1}+F_{2p+4}^2F^2_{\ell-1}F^2_{\ell}+F^2_{2p+4}F_{2\ell}F^2_{\ell-1})}{F_{2p+2}F_{2p+4}F_{2\ell}F_{2\ell+2}F_{2m+2}}\\
		&\quad-\frac{F_{2\ell}F_{2p+4}(F_{2\ell+2}^2F^2_{p+1}F^2_{p}+F_{2\ell+2}F^2_{p+1}F_{2p+2}F^2_{\ell+1} + F^2_{2\ell+2}F^2_{p+1}F_{2p+2})}{F_{2p+2}F_{2p+4}F_{2\ell}F_{2\ell+2}F_{2m+2}}\\
		&\quad-\frac{F_{2\ell}F_{2p+4}(F_{2p+2}F^2_{\ell}F_{2\ell+2}F^2_{p}+F_{2p+2}^2F^2_{\ell}F^2_{\ell+1}+F^2_{2p+2}F_{2\ell+2}F^2_{\ell})}{F_{2p+2}F_{2p+4}F_{2\ell}F_{2\ell+2}F_{2m+2}}.\end{aligned}\]
		Grouping those terms in the numerator containing the product $F_{2p+2}F_{2p+4}F_{2\ell}F_{2\ell+2}$ and canceling the common factors in the numerator and denomenator
\begin{align*}A&=\frac{F^2_{\ell-1}F^2_{p+1}+F_{2p+4}F^2_{\ell-1}+F^2_{p+2}F^2_{\ell}+F_{2\ell}F^2_{p+2} - F^2_{\ell}F^2_p-F_{2p+2}F^2_{\ell}-F^2_{p+1}F^2_{\ell+1}-F_{2\ell+2}F^2_{p+1}}{F_{2m+2}}\\
&=\frac{F^2_{\ell-1}(F^2_{p+1}+F_{2p+4})+F^2_{\ell}(F^2_{p+2}-F^2_p-F_{2p+2}) + F_{2\ell}F^2_{p+2}-F^2_{\ell+1}F^2_{p+1}-F_{2\ell+2}F^2_{p+1}}{F_{2m+2}}\\
&=\frac{F^2_{p+1}(F^2_{\ell-1}-F^2_{\ell+1}-F_{2\ell+2})+F^2_{\ell-1}F_{2p+4}+F_{2\ell}F^2_{p+2}}{F_{2m+2}}\\
&=\frac{-F^2_{p+1}F_{2\ell}-F^2_{p+1}F_{2\ell+2}+F^2_{\ell-1}F_{2p+4}+F_{2\ell}F^2_{p+2}}{F_{2m+2}}.\end{align*}
		Grouping those terms in the numerator containing the product $F_{2\ell}F_{2\ell+2}$ yields
\begin{align*}B&=\frac{F_{2\ell}F_{2\ell+2}F^2_{p+1}(F^2_{p+2}[F_{2p+2\ell+4}+F_{2p+2}F_{2\ell}]-F_{2\ell+2}F^2_pF_{2p+4})}{	F_{2p+2}F_{2p+4}F_{2\ell}F_{2\ell+2}F_{2m+2}}\\
&=\frac{F_{2\ell}F_{2\ell+2}F^2_{p+1}(F^2_{p+2}[F_{2p+2}F_{2\ell+1}+F_{2p+3}F_{2\ell+2}+F_{2p+2}F_{2\ell}]-F_{2\ell+2}F^2_pF_{2p+4})}{F_{2p+2}F_{2p+4}F_{2\ell}F_{2\ell+2}F_{2m+2}}\\	
&=\frac{F_{2\ell}F_{2\ell+2}F^2_{p+1}(F^2_{p+2}[F_{2p+2}F_{2\ell+2}+F_{2p+3}F_{2\ell+2}]-F_{2\ell+2}F^2_pF_{2p+4})}{F_{2p+2}F_{2p+4}F_{2\ell}F_{2\ell+2}F_{2m+2}}\\	
&=\frac{F_{2\ell}F_{2\ell+2}F^2_{p+1}(F^2_{p+2}F_{2p+4}F_{2\ell+2}-F_{2\ell+2}F^2_pF_{2p+4})}{F_{2p+2}F_{2p+4}F_{2\ell}F_{2\ell+2}F_{2m+2}}\\	
&=\frac{F_{2\ell}F_{2\ell+2}F^2_{p+1}F_{2p+2}F_{2p+4}F_{2\ell+2}}{F_{2p+2}F_{2p+4}F_{2\ell}F_{2\ell+2}F_{2m+2}}\\	
&=\frac{F^2_{p+1}F_{2\ell+2}}{F_{2m+2}}.\end{align*}	
Finally grouping the remaining terms yields
\begin{align*}C&=\frac{F_{2p+2}F_{2p+4}F^2_{\ell}(F_{2\ell+2}F_{2p+4}F^2_{\ell-1}-F^2_{\ell+1}[F_{2p+2\ell+4}+F_{2p+2}F_{2\ell}])}{	F_{2p+2}F_{2p+4}F_{2\ell}F_{2\ell+2}F_{2m+2}}\\
&=\frac{F_{2p+2}F_{2p+4}F^2_{\ell}(F_{2\ell+2}F_{2p+4}F^2_{\ell-1}-F^2_{\ell+1}F_{2p+4}F_{2\ell+2})}{F_{2p+2}F_{2p+4}F_{2\ell}F_{2\ell+2}F_{2m+2}}\\
&=\frac{-F_{2p+2}F_{2p+4}F^2_{\ell}F_{2\ell+2}F_{2p+4}F_{2\ell}}{F_{2p+2}F_{2p+4}F_{2\ell}F_{2\ell+2}F_{2m+2}}\\
&=\frac{-F_{2p+4}F^2_{\ell}}{F_{2m+2}}.\end{align*}
Thus,
\[\begin{aligned}r_{m,k+1}(1,n)-r_{m,k}(1,n) &=A+B+C\\
&=\frac{-F^2_{p+1}F_{2\ell}-F^2_{p+1}F_{2\ell+2}+F^2_{\ell-1}F_{2p+4}+F_{2\ell}F^2_{p+2}+F^2_{p+1}F_{2\ell+2}-F_{2p+4}F^2_{\ell}}{F_{2m+2}}\\
&=\frac{F_{2\ell}(F^2_{p+2}-F^2_{p+1})+F_{2p+4}(F^2_{\ell-1}-F_{\ell}^2)}{F_{2m+2}}\\
&=\frac{F_{2\ell}F_{p+3}F_p-F_{2p+4}F_{\ell-2}F_{\ell+1}}{F_{2m+2}}\\
&=\frac{F_{2m-2k+2}F_{k+1}F_{k-2}-F_{2k}F_{m-k-1}F_{m-k+2}}{F_{2m+2}}\\
&=\frac{(-1)^{k+1}F_{m-2k+1}(F_{m+2}+F_{k-1}F_{m-k})}{F_{2m+2}} \quad\text{(By Lemma~\ref{lem35}).}
\end{aligned}\]
Adding this result to $\displaystyle \frac{m+1}{5}+\frac{4F_{m+1}}{5L_{m+1}}+\frac{\sum_{j=3}^{k}\left[(-1)^jF_{m-2j+3}(F_{m+2}+F_{j-2}F_{m-j+1}\right]}{F_{2m+2}}$ yields

$\displaystyle \frac{m+1}{5}+\frac{4F_{m+1}}{5L_{m+1}}+\frac{\sum_{j=3}^{k+1}\left[(-1)^jF_{m-2j+3}(F_{m+2}+F_{j-2}F_{m-j+1}\right]}{F_{2m+2}}$ which is $r_{m,k+1}(1,n)$ as desired.

\end{proof}
\section{Appendix}\label{app}
\begin{lemma}\label{lem:nick}
\[2F_{2m-2}F_{m+1}^3+F_{2m+2}F_{m-1}^2F_m=L_m(F_{m+1}F_{2m-2}F_{m+2}+F_{m}F_{m-1}^2F_{m+1})\]
for any integer $m \geq 1$.
\end{lemma}
\begin{proof}
\begin{align*}
&2F_{2m-2}F_{m+1}^3+F_{2m+2}F_{m-1}^2F_m\\ 
&\quad=2F_{m-1}L_{m-1}F_{m+1}^3 +F_{m+1}L_{m+1}F_{m-1}^2F_m \\
&\quad=F_{m+1}(2F_{m-1}L_{m-1}F_{m+1}^2 +L_{m+1}F_{m-1}^2F_m)\\
&\quad=F_{m+1}(2F_{m-1}L_{m-1}F_{m+1}^2 +L_{m}F_{m-1}^2F_m+L_{m-1}F_{m-1}^2F_m)\\
&\quad=F_{m+1}L_{m-1}(2F_{m-1}F_{m+1}^2 +F_{m-1}^2(F_{m+1}-F_{m-1}))+L_{m}F_{m-1}^2F_mF_{m+1}\\
&\quad=F_{m+1}L_{m-1}(F_{m-1}(F_{m+1}^2 -F_{m-1}^2) +F_{m-1}F_{m+1}^2+F_{m-1}^2F_{m+1})+L_{m}F_{m-1}^2F_mF_{m+1}\\
&\quad=F_{m+1}L_{m-1}(F_{m-1}F_{2m} +F_{m-1}F_{m+1}(F_{m+1}+F_{m-1}))+L_{m}F_{m-1}^2F_mF_{m+1}\\
&\quad=F_{m+1}L_{m-1}(F_{m-1}F_mL_m +F_{m-1}F_{m+1}L_m)+L_{m}F_{m-1}^2F_mF_{m+1}\\
&\quad=L_m(F_{m+1}L_{m-1}F_{m-1}F_{m+2}+F_mF_{m-1}^2F_{m+1})\\
&\quad=L_m(F_{m+1}F_{2m-2}F_{m+2}+F_{m}F_{m-1}^2F_{m+1})
\end{align*}
\end{proof}
\begin{lemma}\label{lema}
\[F_{m+1}F_{m+2}-3F_{m-3}F_m = 2F_{2m-3}+3F_{m+1}F_{m-2}+F_{m}F_{m-1}\]
for integers $m \geq 1$.
\end{lemma}
\begin{proof}
\begin{align*}
&F_{m+1}F_{m+2}-3F_{m-3}F_m\\
&\quad = F_{2m+1}-3F_mF_{m-3}+F_mF_{m-1}\\
&\quad =F_{2m-2}+2F_{2m-1}-3F_mF_{m-3}+F_mF_{m-1}\\
&\quad =2F_{2m-3}+3(F_{2m-2}-F_mF_{m-3})+F_mF_{m-1}\\
&\quad = 2F_{2m-3}+3F_{m+1}F_{m-2}+F_{m}F_{m-1}
\end{align*}
\end{proof}
\begin{lemma}\label{lemb}
\[(F_{2m-2}+3F_{m-2}^2)(4F_{2m-2}+3F^2_{m-1})+F_{2m-2}F_{2m+2} = 3\left[F_{2m-2}(2F_{2m-3}+3F_{m+1}F_{m-2}+F_{m}F_{m-1})+F_mF_{m+1}F_{m-1}^2\right]\]
for integers $m \geq 1$.
\end{lemma}
\begin{proof}
\begin{align*}
&(F_{2m-2}+3F_{m-2}^2)(4F_{2m-2}+3F^2_{m-1})+F_{2m-2}F_{2m+2}\\
&\quad = 4F_{2m-2}^2+F_{2m-2}F_{2m+2}+12F^2_{m-2}F_{2m-2}+3F_{m-1}^2F_{2m-2}+9F_{m-2}^2F^2_{m-1}\\
&\quad = 3F_{2m-2}^2+3F_{2m-2}F_{2m}+12F^2_{m-2}F_{2m-2}+3F_{m-1}^2F_{2m-2}+9F_{m-2}^2F^2_{m-1}\\
&\quad = 3\left[F_{2m-2}^2+F_{2m-2}F_{2m}+4F^2_{m-2}F_{2m-2}+F_{m-1}^2F_{2m-2}+3F_{m-2}^2F^2_{m-1}\right]\\
&\quad = 3\left[F_{2m-2}(F_{2m-2}+F_{2m}+4F^2_{m-2}+F_{m-1}^2)+3F_{m-2}^2F^2_{m-1}\right]\\
&\quad = 3\left[F_{2m-2}(F_{2m-2}+F_{2m}+3F^2_{m-2}+F_{2m-3})+3F_{m-2}^2F^2_{m-1}\right]\\
&\quad = 3\left[F_{2m-2}(F_{2m+1}+3F^2_{m-2})+3F_{m-2}^2F^2_{m-1}\right]\\
&\quad = 3\left[F_{2m-2}(F_m^2+F_{m+1}^2+3F^2_{m-2})+3F_{m-2}^2F^2_{m-1}\right]\\
&\quad = 3\left[F_{2m-2}(F_m^2+2F_{m-2}^2+2(F^2_{m-1}+F^2_m))+3F_{m-2}^2F^2_{m-1}\right]\\
&\quad = 3\left[F_{2m-2}(3F_m^2+2F_{2m-3})+3F_{m-2}^2F^2_{m-1}\right]\\
&\quad = 3\left[F_{2m-2}(3F_m^2+2F_{2m-3})+3(F_m^2-F_{2m-2})F^2_{m-1}\right]\\
&\quad = 3\left[F_{2m-2}(3(F_m^2-F_{m-1}^2)+2F_{2m-3})+2F_m^2F^2_{m-1}-F_{m-1}^3F_m+F_mF_{m+1}F_{m-1}^2\right]\\
&\quad = 3\left[F_{2m-2}(3(F_m^2-F_{m-1}^2)+2F_{2m-3})+F_m^2F^2_{m-1}+F_mF_{m-1}^2(F_m-F_{m-1})+F_mF_{m+1}F_{m-1}^2\right]\\
&\quad = 3\left[F_{2m-2}(3(F_m^2-F_{m-1}^2)+2F_{2m-3})+F_mF_{m-1}^2(F_m+F_{m-2})+F_mF_{m+1}F_{m-1}^2\right]\\
&\quad = 3\left[F_{2m-2}(3(F_m^2-F_{m-1}^2)+2F_{2m-3})+F_mF_{m-1}^2L_{m-1}+F_mF_{m+1}F_{m-1}^2\right]\\
&\quad = 3\left[F_{2m-2}(3(F_m^2-F_{m-1}^2)+2F_{2m-3})+F_mF_{m-1}F_{2m-2}+F_mF_{m+1}F_{m-1}^2\right]\\
&\quad = 3\left[F_{2m-2}(3(F_m^2-F_{m-1}^2)+2F_{2m-3}+F_mF_{m-1})+F_mF_{m+1}F_{m-1}^2\right]\\
&\quad = 3\left[F_{2m-2}(3F_{m+1}F_{m-2}+2F_{2m-3}+F_mF_{m-1})+F_mF_{m+1}F_{m-1}^2\right]
\end{align*}
\end{proof}
\begin{lemma}\label{base}
For integer values of $m$ greater than 3
\begin{align*}
&\frac{m+1}{5}+\frac{4F_{m+1}}{5L_{m+1}} -\frac{F_{m-3}(F_{m+2}+F_{m-2})}{F_{2m+2}}=\\
&\quad\quad = \frac{(F_{2m-2}+3F_{m-2}^2)(4F_{2m-2}+3F^2_{m-1})}{3F_{2m-2}F_{2m+2}}+\frac{1}{3}+\sum_{i=1}^{m-2}\frac{F_iF_{i+1}}{L_iL_{i+1}}.\end{align*}
\end{lemma}
\begin{proof}
By~\cite[Proposition 52]{bef}
\[
\frac{2F_{m+1}^2}{L_m L_{m+1}} + \frac{mL_{m} - F_{m}}{5 L_{m}}= \frac{m+1}{5} +  \frac{4F_{m+1}}{5L_{m+1}},
\]
and by~\cite[Proposition 51]{bef}
\[
 \sum_{i = 1}^{m-1} \frac{F_i F_{i+1}}{L_i L_{i+1}} = \frac{m L_{m} - F_{m}}{5 L_{m}}.
\]
Hence (applying Lemmas~\ref{lem:nick}-\ref{lemb})
\begin{align*}
&\frac{m+1}{5}+\frac{4F_{m+1}}{5L_{m+1}} -\frac{F_{m-3}(F_{m+2}+F_{m-2})}{F_{2m+2}}\\
&\quad = \frac{2F_{m+1}^2}{L_m L_{m+1}}+\frac{F_{m-1}F_m}{L_{m-1}{L_m}}  -\frac{F_{m-3}(F_{m+2}+F_{m-2})}{F_{2m+2}}+   \sum_{i = 1}^{m-2} \frac{F_i F_{i+1}}{L_i L_{i+1}}\\
&\quad = \frac{2F_{m+1}^3}{L_m F_{2m+2}}+\frac{F_{m-1}^2F_m}{F_{2m-2}{L_m}}  -\frac{F_{m-3}(F_{m+2}+F_{m-2})}{F_{2m+2}}+   \sum_{i = 1}^{m-2} \frac{F_i F_{i+1}}{L_i L_{i+1}}\\
&\quad = \frac{2F_{2m-2}F_{m+1}^3+F_{2m+2}F_{m-1}^2F_m}{L_m F_{2m+2}F_{2m-2}}  -\frac{F_{m-3}(F_{m+2}+F_{m-2})}{F_{2m+2}}+   \sum_{i = 1}^{m-2} \frac{F_i F_{i+1}}{L_i L_{i+1}}\\
&\quad = \frac{L_m(F_{m+1}F_{2m-2}F_{m+2}+F_{m}F_{m-1}^2F_{m+1})}{L_m F_{2m+2}F_{2m-2}}  -\frac{F_{m-3}(F_{m+2}+F_{m-2})}{F_{2m+2}}+   \sum_{i = 1}^{m-2} \frac{F_i F_{i+1}}{L_i L_{i+1}} \quad\text{(By Lemma~\ref{lem:nick})}\\
&\quad = \frac{F_{m+1}F_{2m-2}F_{m+2}+F_{m}F_{m-1}^2F_{m+1}- F_{m-3}F_{2m-2}(F_{m+2}+F_{m-2})}{ F_{2m+2}F_{2m-2}} +   \sum_{i = 1}^{m-2} \frac{F_i F_{i+1}}{L_i L_{i+1}}\\
&\quad = \frac{F_{2m-2}(F_{m+1}F_{m+2}- F_{m-3}(F_{m+2}+F_{m-2}))+F_{m}F_{m-1}^2F_{m+1}}{ F_{2m+2}F_{2m-2}} +   \sum_{i = 1}^{m-2} \frac{F_i F_{i+1}}{L_i L_{i+1}}\\
&\quad = \frac{F_{2m-2}(F_{m+1}F_{m+2}- 3F_{m-3}F_m)+F_{m}F_{m-1}^2F_{m+1}}{ F_{2m+2}F_{2m-2}} +   \sum_{i = 1}^{m-2} \frac{F_i F_{i+1}}{L_i L_{i+1}}\\
&\quad = \frac{F_{2m-2}(2F_{2m-3}+3F_{m+1}F_{m-2}+F_{m}F_{m-1})+F_{m}F_{m-1}^2F_{m+1}}{ F_{2m+2}F_{2m-2}} +   \sum_{i = 1}^{m-2} \frac{F_i F_{i+1}}{L_i L_{i+1}} \quad\text{(By Lemma~\ref{lema})}\\
&\quad = \frac{3\left[F_{2m-2}(2F_{2m-3}+3F_{m+1}F_{m-2}+F_{m}F_{m-1})+F_{m}F_{m-1}^2F_{m+1}\right]}{ 3F_{2m+2}F_{2m-2}} +   \sum_{i = 1}^{m-2} \frac{F_i F_{i+1}}{L_i L_{i+1}}\\
&\quad = \frac{(F_{2m-2}+3F_{m-2}^2)(4F_{2m-2}+3F^2_{m-1})+F_{2m-2}F_{2m+2}}{ 3F_{2m+2}F_{2m-2}} +   \sum_{i = 1}^{m-2} \frac{F_i F_{i+1}}{L_i L_{i+1}} \quad\text{(By Lemma~\ref{lemb})}\\
&\quad = \frac{(F_{2m-2}+3F_{m-2}^2)(4F_{2m-2}+3F^2_{m-1})}{3F_{2m-2}F_{2m+2}}+\frac{1}{3}+\sum_{i=1}^{m-2}\frac{F_iF_{i+1}}{L_iL_{i+1}}.
\end{align*}
\end{proof}

\bibliography{references}{}
\bibliographystyle{plain} 

\end{document}